\documentclass[12pt, a4paper]{article}

\usepackage{mathptmx}
\usepackage{amsmath}
\usepackage{amsthm}
\usepackage[colorlinks,citecolor=blue,urlcolor=blue,linkcolor=blue,pagecolor=blue,linktocpage=true]{hyperref}
\usepackage{amssymb}
\usepackage{graphicx}
\usepackage[tight]{subfigure}
\usepackage[round]{natbib}

\newcommand{\iid}{i.i.d.}
\DeclareMathOperator{\E}{E}
\DeclareMathOperator{\Var}{Var}
\newcommand{\diff}{\mathrm d}
\newcommand{\vahp}{\hat p}
\newcommand{\vahd}{\hat d}
\newcommand{\odds}{\gamma}
\newcommand{\vahodds}{\hat \odds}
\newcommand{\logodds}{\phi}
\newcommand{\vahlogp}{\hat \lambda}
\newcommand{\vahlogodds}{\hat \logodds}
\newcommand{\nsuc}{r}
\newcommand{\vansam}{N}
\newcommand{\nsam}{n}
\newcommand{\kvar}{k}
\newcommand{\effic}{\eta}
\newcommand{\vab}{X}
\newcommand{\fgen}{g}
\newcommand{\paramgen}{\theta}

\newcommand{\vahparamgen}{\hat \theta}
\newcommand{\coefdg}{c}
\newcommand{\harm}{H}
\newcommand{\hasbound}{\alpha}
\newcommand{\vertdiff}{\beta}
\newcommand{\hasgen}{\zeta}

\newtheorem{proposition}{Proposition}
\newtheorem{theorem}{Theorem}
\newtheorem{lemma}{Lemma}

\graphicspath{{figures/}}

\begin{document}

\title{
Estimating odds and log odds with guaranteed accuracy
}

\author{Luis Mendo\\
Universidad Polit\'ecnica de Madrid, Spain\\
\texttt{luis.mendo@upm.es}\\
}

\maketitle

\begin{abstract}
Two sequential estimators are proposed for the odds $p/(1-p)$ and log odds $\log(p/(1-p))$ respectively, using independent Bernoulli random variables with parameter $p$ as inputs. The estimators are unbiased, and guarantee that the variance of the estimation error divided by the true value of the odds, or the variance of the estimation error of the log odds, are less than a target value for any $p \in (0,1)$. The estimators are close to optimal in the sense of Wolfowitz's bound.


\end{abstract}

\section{Introduction}
\label{part: intro}

Consider a probability $p \in(0,1)$, and let $\odds = p/(1-p)$ and $\logodds = \log\odds$ respectively denote the \emph{odds} and \emph{log odds} associated with $p$. These are commonly used measures of the probability of an event, which find application in many fields including clinical trials, epidemiology, economics, sociology or data analysis. Also, log odds are inherent to logistic regression, which is a prevalent tool in medical and social sciences \citep{Agresti02, Norton18},
and is widely used in machine learning \citep{Bishop06}.

For an estimator $\vahodds$ of $\odds$, a usual measure of accuracy is the \emph{mean-square error} (MSE), $\E[(\vahodds-\odds)^2]$, or its square root, known as root-mean-square error (RMSE). Either of these measures of error is most meaningful in a \emph{relative} sense. For example, an RMSE equal to $0.5$ may be adequate if the true odds are $\odds = 9$ (corresponding to $p=0.9$), but unacceptable if $\odds = 0.25$ ($p=0.2$). It is reasonable to require that the RMSE be proportional to the true (unknown) value of $\odds$; or equivalently that the MSE be proportional to $\odds^2$. The estimator $\vahodds$ should ideally guarantee $\E[(\vahodds-\odds)^2] / \odds^2$ smaller than a \emph{prescribed} value, irrespective of $\odds$.

When estimating the log odds $\logodds$, on the other hand, the error should not be understood in a relative sense, because the logarithm already has a normalizing effect, as it transforms ratios into differences. Specifically, a difference $\delta$ between $\logodds$ and an estimate $\vahlogodds$ thereof corresponds to a deviation by a factor $\exp \delta$ between $\odds$ and its estimate $\vahodds$ defined as $\exp \vahlogodds$. That is, relative error in estimating $\odds$ corresponds to absolute (non-normalized) error in estimating $\logodds$.

This paper deals with the problem of estimating $\odds$ and $\logodds$ from independent, identically distributed (\iid)~observations of a Bernoulli distribution with parameter $p$. Sequential estimators of $\odds$ and $\logodds$ are presented that achieve a specified level of accuracy, defined as relative or absolute MSE respectively. The sequential character of the estimators cannot be avoided, because a fixed sample size cannot satisfy these conditions for all $p$. The proposed estimators are based on \emph{inverse binomial sampling} \citetext{\citealp{Haldane45}; \citealp{Mikulski76}; \citealp[chapter~2]{Lehmann98}}, and have a number of desirable features, in addition to guaranteeing a target accuracy: they are simple to compute; they are unbiased; and their performance is close to optimal, in the sense that the relationship between accuracy and average sample size is only slightly worse than that given by Wolfowitz's bound \citep{Wolfowitz47}.

The problems of estimating the odds and the log odds, as well as related parameters such as the ratio of two probabilities or the corresponding odds ratio,
have drawn a great deal of attention over the years. However, to the author's knowledge, no estimator has been proposed that can guarantee a specified level of accuracy as discussed above, i.e.~in a relative sense for the odds or in an absolute sense for the log odds, for all values of $p$. A review of existing results in this area is given next.

Early works by \citet{Gart67} and \citet{Wells87} focus on analysing the bias and mean-square error of several estimators of the odds under binomial (i.e.~fixed-size) sampling. \citet{Siegmund82} studies the asymptotic properties of estimators of the odds and the odds ratio. \citet{Sungboonchoo21} propose estimators of the odds with exponentially decreasing bias, using either binomial or inverse binomial sampling; and provide asymptotic confidence intervals for estimators of the odds ratio. \citet{Ngamkham16} and \citet{Ngamkham18} consider interval estimation of the ratio of two probabilities, when each of the two populations is sampled independently using either a fixed sample size or inverse binomial sampling, and give asymptotic expressions for confidence intervals with a desired coverage probability. Similarly, \citet{Tian08} discuss the estimation of the ratio of two probabilities using independent inverse binomial sampling for each population, and review several approaches for obtaining approximate confidence intervals.
\citet{Cho07} estimates the probability ratio under a loss function defined as the sum of squared error and cost proportional to the number of observations, with samples taken in pairs, one from each population; and
proposes a sequential procedure which is asymptotically optimal when the cost per sample is small. \citet{Cho13} presents a sequential estimator for a probability ratio when the proportion of sample sizes is specified, and studies its asymptotic properties. \citet{Kokaew21, Kokaew23} consider correlation between observations of the two populations, assuming sample sets of fixed size; and propose estimators of the probability ratio or its logarithm, for which they derive asymptotic confidence intervals. \citet{Agresti99} analyses confidence intervals for the log odds ratio when independent binomial sampling is applied to each population, and \citet{Bandyopadhyay17} address the more general setting where the two populations are sampled using different combinations of binomial and inverse binomial sampling.

The following terminology and simple facts will be used throughout the paper. The two possible values of the Bernoulli distribution will be referred to as ``success'' and ``failure'', which occur with probabilities $p$ and $1-p$ respectively. In inverse binomial sampling, given $\nsuc \in \mathbb N$, a sequence of \iid~Bernoulli variables with parameter $p$ is observed until the $\nsuc$-th success is obtained. The required number of samples, $\vansam$, has a \emph{negative binomial} distribution:
\begin{equation}
\Pr[\vansam = \nsam] = \binom{\nsam-1}{\nsuc-1} p^\nsuc (1-p)^{\nsam-\nsuc},\ \ \ \nsam \geq \nsuc.
\end{equation}
(An alternative definition, not used in this paper, considers the negative binomial distribution as that of $\vansam-\nsuc$.) For this distribution \citetext{\citealp{Haldane45}; \citealp[chapter~4]{Ross10}},
\begin{align}
\label{eq: neg bin: E vansam}
\E[\vansam] &= \frac{\nsuc}{p}, \\
\label{eq: neg bin: E inv minus 1}
\E\left[\frac 1{\vansam-1}\right] &= \frac{p}{\nsuc-1}, \\
\label{eq: neg bin: Var vansam}
\Var[\vansam] &= \frac{\nsuc (1-p)}{p^2}.
\end{align}

\section{Estimation of odds}
\label{part: odds}

The estimator of $\odds = p/(1-p)$ to be proposed is based on estimating $p$ and $1/(1-p)$ separately by means of inverse binomial sampling. Specifically, given $\nsuc_1 \in \mathbb N$, $\nsuc_1 \geq 3$, a random number $\vansam_1$ of samples are observed until $\nsuc_1$ \emph{successes} occur. Then $\vahp_1 = (\nsuc_1-1)/(\vansam_1-1)$ is an unbiased estimator of $p$ as a consequence of \eqref{eq: neg bin: E inv minus 1}, and its variance satisfies the bound \citep{Pathak84}
\begin{equation}
\label{eq: IBS: Var hatp}
\Var[\vahp_1] \leq \frac{p^2 (1-p)}{\nsuc_1-2+2p}.
\end{equation}

Similarly, for $\nsuc_2 \in \mathbb N$, a second, independent set of samples is taken until $\nsuc_2$ \emph{failures} are observed. The number of samples $\vansam_2$ is a negative binomial random variable with parameters $\nsuc_2$ and $1-p$. Replacing $\nsuc$ and $p$ in \eqref{eq: neg bin: E vansam} and \eqref{eq: neg bin: Var vansam} by $\nsuc_2$ and $1-p$, it follows that $\vahd_2 = \vansam_2/\nsuc_2$ is an unbiased estimator of $1/(1-p)$, and
\begin{equation}
\label{eq: FS: Var hatd}
\Var[\vahd_2] = \frac{\Var[\vansam_2]}{\nsuc_2^2} = \frac{p}{\nsuc_2 (1-p)^2}.
\end{equation}

Given $\nsuc \in \mathbb N$, $\nsuc \geq 2$, the estimator of $\odds$ is defined as
\begin{equation}
\label{eq: vahodds}
\vahodds = \vahp_1 \vahd_2
\end{equation}
where $\vahp_1$ and $\vahd_2$ are obtained as above with $\nsuc_1 = \nsuc+1$ and $\nsuc_2 = \nsuc-1$. Let $\vansam = \vansam_1 + \vansam_2$ denote the total number of samples used by this estimator.

\begin{theorem}
\label{theo: vahodds}
For $\nsuc \in \mathbb N$, $\nsuc \geq 2$, $p \in (0,1)$, the estimator $\vahodds$ given by \eqref{eq: vahodds} has the following properties:
\begin{align}
\label{eq: E vahodds unbiased}
\E[\vahodds] &= \odds \\
\label{eq: Var vahodds bound}
\frac{\Var[\vahodds]}{\odds^2} &\leq \frac 1 {\nsuc-1} \left( 1 - \frac{p(1-p)}{\nsuc-1+2p} \right) \\
\label{eq: Var vahodds bound guar}
&< \frac 1 {\nsuc-1} \\[1mm] 
\label{eq: vahodds sample size}
\E[\vansam] &= \frac{\nsuc+1-2p}{p(1-p)}.
\end{align}
\end{theorem}

\begin{proof}
The estimators $\vahp_1$ and $\vahd_2$ are statistically independent, because they are based on two separate sets of samples. Therefore
\begin{equation}
\E[\vahodds] = \E[\vahp_1] \E[\vahd_2] = \frac p {1-p} = \odds,
\end{equation}
which proves \eqref{eq: E vahodds unbiased}.

Similarly, using \eqref{eq: IBS: Var hatp} and \eqref{eq: FS: Var hatd} with $\nsuc_1 = \nsuc+1$, $\nsuc_2 = \nsuc-1$,
\begin{equation}
\begin{split}
\Var[\vahodds] &= \E[\vahp_1^2] \E[\vahd_2^2] - \frac{p^2}{(1-p)^2} \\
&= \left( \Var[\vahp_1] + p^2  \right) \left( \Var[\vahd_2] + \frac 1 {(1-p)^2} \right) - \frac{p^2}{(1-p)^2} \\
&\leq \left( \frac{p^2 (1-p)}{\nsuc-1+2p} + p^2 \right) \left( \frac{p}{(\nsuc-1) (1-p)^2} + \frac 1 {(1-p)^2} \right) - \frac{p^2}{(1-p)^2} \\
&= \frac{p^2}{(1-p)^2} \left( \frac{p(1-p)}{(\nsuc-1+2p)(\nsuc-1)} + \frac{1-p}{\nsuc-1+2p} + \frac p {\nsuc-1} \right) \\
&= \frac{(\nsuc-1+p+p^2)p^2}{(\nsuc-1)(\nsuc-1+2p)(1-p)^2} \\
&= \frac{p^2}{(\nsuc-1)(1-p)^2} \left(1 - \frac{p(1-p)}{\nsuc-1+2p} \right) < \frac {p^2} {(\nsuc-1)(1-p)^2},
\end{split}
\end{equation}
as in \eqref{eq: Var vahodds bound} and \eqref{eq: Var vahodds bound guar}.

The average number of samples is obtained as
\begin{equation}
\E[\vansam] = \E[\vansam_1] + \E[\vansam_2] = \frac{\nsuc+1} p + \frac{\nsuc-1} {1-p} = \frac{\nsuc+1-2p}{p(1-p)},
\end{equation}
in agreement with \eqref{eq: vahodds sample size}.
\end{proof}

From Theorem~\ref{theo: vahodds} it stems that a relative error as low as desired can be guaranteed, irrespective of $p$, by choosing $\nsuc$ large enough. For example, to achieve an RMSE smaller than $20\%$ of the true value of $\odds$
it is sufficient to take $\nsuc = 26$, according to \eqref{eq: Var vahodds bound guar}. Guaranteeing better accuracy values requires larger $\nsuc$, which in turn causes the average sample size to increase, as given by \eqref{eq: vahodds sample size}. As a continuation of the example, decreasing the target RMSE by a factor of $1/\sqrt{2}$, i.e.~reducing the target MSE to one half, requires an increase of $\nsuc$ from $26$ to $51$, with the result that $\E[\vansam]$ is approximately doubled.

To assess how good the trade-off between accuracy and average sample size is, it is meaningful to use Wolfowitz's bound, which is a generalization of the Cram\'er-Rao inequality for sequential estimators \citetext{\citealp{Wolfowitz47}; \citealp[section~4.3]{Ghosh97}}. The bound indicates the lowest variance that an estimator can have for a given average sample size, provided that certain (mild) regularity conditions are satisfied by the distribution function of the observations, the sequential procedure and the estimator. When the observations are \iid~Bernoulli random variables with parameter $p$ and the parameter to be estimated, $\paramgen$, is a differentiable function of $p$, Wolfowitz's bound for an unbiased estimator $\vahparamgen$ is \citep[lemma~2.7]{DeGroot59}
\begin{equation}
\label{eq: Wolfowitz Bernoulli unbiased}
\Var[\vahparamgen] \E[\vansam] \geq \left(\displaystyle \frac{\diff \paramgen}{\diff p}\right)^2 p (1-p).
\end{equation}
The \emph{efficiency} of $\vahparamgen$, which will be denoted as $\effic(\vahparamgen)$, can then be defined as Wolfowitz's bound for $\Var[\vahparamgen] \E[\vansam]$ divided by the actual value of that product achieved by $\vahparamgen$:
\begin{equation}
\label{eq: effic vahparamgen}
\effic(\vahparamgen) = \frac{({\diff \paramgen}/{\diff p})^2\, p (1-p)}{\Var[\vahparamgen] \E[\vansam]}.
\end{equation}
Taking $\paramgen = \odds$, $\vahparamgen = \vahodds$, which implies that $\diff \paramgen / \diff p = 1/(1-p)^2$ and that $\Var[\vahodds]$ and $\E[\vansam]$ satisfy \eqref{eq: Var vahodds bound} and \eqref{eq: vahodds sample size} respectively,  \eqref{eq: effic vahparamgen} becomes
\begin{align}
\label{eq: effic vahodds eq}
\effic(\vahodds) &= \frac{p}{(1-p)^3 \Var[\vahodds] \E[\vansam]} \\
\label{eq: effic vahodds bound 1}
& > \frac {\nsuc-1} {\nsuc+1-2p} \left(1 + \frac{p(1-p)}{\nsuc-1+p+p^2} \right) \\
\label{eq: effic vahodds bound 2}
& > \frac {\nsuc-1} {\nsuc+1}.
\end{align}
Thus, the efficiency of $\vahodds$ approaches $1$ for large $\nsuc$.



Figure~\ref{fig: odds_effic_p1m1} compares, for several values of $\nsuc$, the lower bound \eqref{eq: effic vahodds bound 1} with the actual efficiency, obtained by Monte Carlo simulation. Specifically, for each~$p$, $10^8$ values of the estimate $\vahodds$ are simulated using Bernoulli random variables with parameter $p$; the sample variance of $\vahodds$ is computed; and this is used instead of the true variance in \eqref{eq: effic vahodds eq} to obtain $\effic(\vahodds)$. As can be seen, the bound is quite tight, especially for $p$ near $0$ or $1$ or for $\nsuc$ not too small.

\begin{figure}%
\centering%
\includegraphics[width=.75\textwidth]{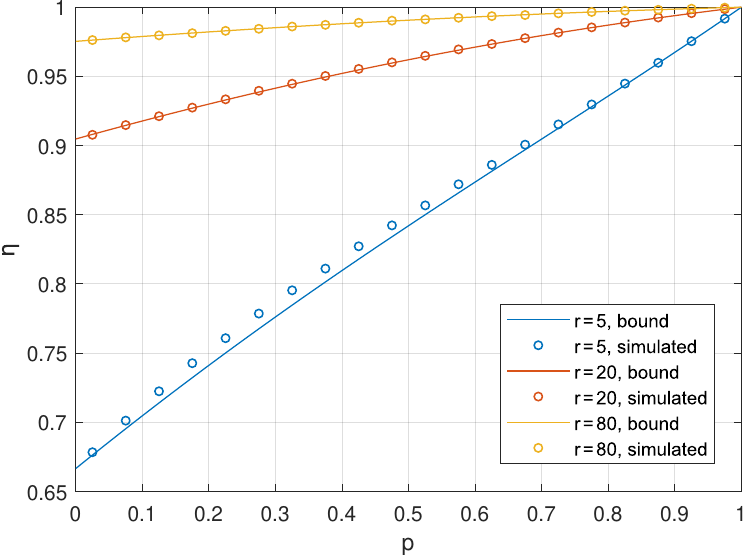}
\caption{Efficiency of odds estimator $\vahodds$}%
\label{fig: odds_effic_p1m1}%
\end{figure}%

\section{Estimation of log odds}
\label{part: log odds}

The estimator of $\logodds = \log(p/(1-p))$, to be defined precisely later, will be based on estimating $\log p$ and $\log (1-p)$ separately using inverse binomial sampling. Thus the estimation of $\log p$ is analysed next.

Given $\nsuc \in \mathbb N$, $\nsuc \geq 2$, and denoting by $\vansam$ the number of samples required to obtain $\nsuc$ successes, an unbiased estimator of $\log p$ can be obtained from $\nsuc$ and $\vansam$ using the following result by \citet[theorem~4.1]{DeGroot59}. Let $q=1-p$. Given a function $\fgen(q)$, there exists an unbiased estimator of $\paramgen = \fgen(q)$ in inverse binomial sampling if and only if $\fgen(q)$ has a Taylor expansion around $q=0$ valid for $|q|<1$; and in that case the estimator $\vahparamgen(\vansam)$ is given by
\begin{align}
\label{eq: vahfgen}
\vahparamgen(\vansam) &= \frac{(\nsuc-1)! \, \coefdg_{\nsuc, \vansam-\nsuc}}{(\vansam-1)!}, \\
\label{eq: coefcdg}
\coefdg_{\nsuc, \kvar} &= \frac{\diff^\kvar}{\diff q^\kvar} \left. \frac{\fgen(q)}{(1-q)^\nsuc} \right\rfloor_{q=0}.
\end{align}

Let the $\nsam$-th harmonic number be denoted as $\harm_\nsam = 1 + 1/2 + \cdots + 1/\nsam$.

\begin{proposition}
\label{prop: vahlogp harm}
For $\fgen(q) = \log(1-q) = \log p$, \eqref{eq: vahfgen} and \eqref{eq: coefcdg} yield
\begin{equation}
\label{eq: DeGroot logp}
\vahparamgen(\vansam) = - \left(\frac 1 {\nsuc} + \frac 1 {\nsuc+1} + \cdots + \frac 1 {\vansam-1} \right) =  -\harm_{\vansam-1} + \harm_{\nsuc-1}.
\end{equation}
\end{proposition}

\begin{proof}
Let $x^{(\kvar)}$ denote the rising factorial, $x^{(\kvar)} = x(x+1)\cdots(x+\kvar-1)$, $x^{(0)} = 1$. Direct computation using \eqref{eq: coefcdg} with $\fgen(q) = \log(1-q)$ shows that $c_{\nsuc, 0} = 0$, and for $\kvar > 0$ this recurrence relation holds:
\begin{equation}
\label{eq: coefdg recur}
\coefdg_{\nsuc, \kvar} = -\nsuc^{(\kvar-1)} + (\nsuc+\kvar-1) \coefdg_{\nsuc, \kvar-1}.
\end{equation}
From \eqref{eq: vahfgen} and \eqref{eq: coefdg recur} it stems that $\vahparamgen(\nsuc) = 0$, and for $\nsam > \nsuc$
\begin{equation}
\begin{split}
\vahparamgen(\nsam) &= \frac {(\nsuc-1)! \, (-\nsuc^{(\nsam-\nsuc-1)} + (\nsam-1) \coefdg_{\nsuc, \nsam-\nsuc-1})} {(\nsam-1)!} \\
&= -\frac {1} {\nsam-1} + \frac {(\nsuc-1)! \, \coefdg_{\nsuc, \nsam-\nsuc-1}} {(\nsam-2)!}
= -\frac {1} {\nsam-1} + {\vahparamgen(\nsam-1)},
\end{split}
\end{equation}
which is equivalent to \eqref{eq: DeGroot logp}.
\end{proof}

By Proposition~\ref{prop: vahlogp harm}, an unbiased estimator $\vahlogp$ of $\log p$ using inverse binomial sampling is
\begin{equation}
\label{eq: vahlogp}
\vahlogp = -\harm_{\vansam-1} + \harm_{\nsuc-1},
\end{equation}
where $\nsuc$ is the number of successes and $\vansam \geq \nsuc$ is the number of observations. The variance of this estimator is bounded as shown next. 

\begin{theorem}
\label{theo: vahlogp}
For $\nsuc \in \mathbb N$, $\nsuc \geq 2$, $p \in (0,1)$, the variance of $\vahlogp$ satisfies
\begin{align}
\label{eq: E vahlogp bound}
\Var[\vahlogp]
&< \frac 1 {\nsuc-1+p} \left( \left( 1 + \frac {p} {2\nsuc-1} - \frac {1} {4\nsuc(2\nsuc-1)}\right) (1-p) + \frac p {4(\nsuc-1)} \right) \\
\label{eq: E vahlogp bound guar}
&< \frac 1 {\nsuc-1}.
\end{align}
\end{theorem}

The proof of this theorem makes use of the following lemma.

\begin{lemma}
\label{lem: diff harm}
For $\nsuc,\, \nsam \in \mathbb N$, $\nsam > \nsuc$,
\begin{equation}
\label{eq: lema diff harm: bound}
\log \frac {\nsam-1/2+\hasbound} {\nsuc-1/2+\hasbound} < \harm_{\nsam-1} - \harm_{\nsuc-1} < \log \frac {\nsam-1/2} {\nsuc-1/2} 
\end{equation}
with
\begin{align}
\label{eq: hasbound eq}
\hasbound &= \frac{-\nsuc + \sqrt{\nsuc^2+1}}{2}, \\
\label{eq: harmbound bound}
0 < \hasbound &< \frac 1 {4\nsuc}.
\end{align}
\end{lemma}

\begin{proof}
Consider $\nsuc,\, \nsam \in \mathbb N$, $\nsam > \nsuc$. The shaded area in Figure~\ref{fig: lem diff harm upper}, where the heights of the rectangles are $1/\nsuc, \ldots, 1/(\nsam-1)$, equals $\harm_{\nsam-1}-\harm_{\nsuc-1}$. Replacing the top side of each rectangle by the straight line that passes through its middle point and is tangent to the curve $y = 1/(x+1/2)$, a trapezium
is obtained with the same area as the rectangle. The figure shows, with dashed line, the trapezium associated with one of the rectangles. It follows from the convexity of the curve that
\begin{equation}
\harm_{\nsam-1}-\harm_{\nsuc-1} < \int_{\nsuc-1}^{\nsam-1} \frac { \diff x} {x+1/2} = \log \frac {\nsam-1/2} {\nsuc-1/2},
\end{equation}
which establishes the second inequality in \eqref{eq: lema diff harm: bound}.

\begin{figure}%
\centering%
\subfigure[Upper bound]{%
\label{fig: lem diff harm upper}%
\includegraphics[width=.7\textwidth]{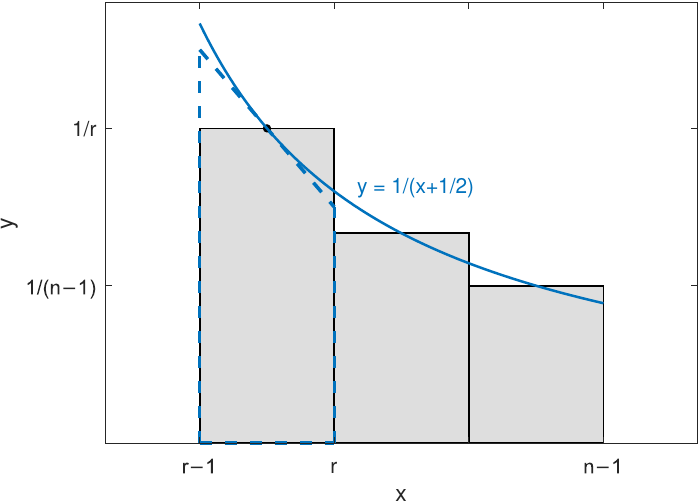}
}\\%
\subfigure[Lower bound]{%
\label{fig: lem diff harm lower}%
\includegraphics[width=.7\textwidth]{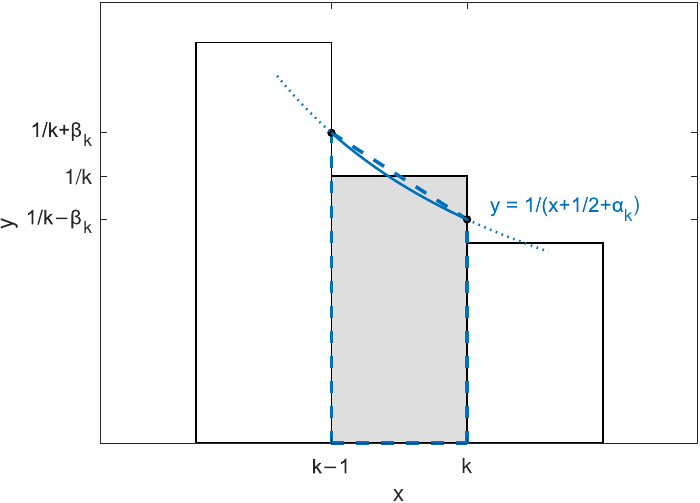}%
}%
\caption{Proof of bounds in \eqref{eq: lema diff harm: bound}}%
\label{fig: lem diff harm}%
\end{figure}%

For the first inequality in \eqref{eq: lema diff harm: bound}, consider the shaded rectangle in Figure~\ref{fig: lem diff harm lower}, which extends from $x=\kvar-1$ to $x=\kvar$ and has height $1/\kvar$, with $\kvar \in \{\nsuc, \ldots, \nsam-1\}$ arbitrary. Let $\hasbound_\kvar > 0$ be chosen so that the intercept points of the curve $y = 1/(x+1/2+\hasbound_\kvar)$ with the vertical lines $x = \kvar-1$ and $x = \kvar$ are vertically equispaced from the horizontal line $y = 1/\kvar$, i.e.~have heights $1/\kvar+\vertdiff_\kvar$ and $1/\kvar-\vertdiff_\kvar$ for some $\vertdiff_\kvar > 0$. That is, $\hasbound_\kvar$ and $\vertdiff_\kvar$ satisfy the equation system
\begin{align}
\frac 1 {\kvar-1/2+\hasbound_\kvar} &= \frac 1 {\kvar} + \vertdiff_\kvar \\
\frac 1 {\kvar+1/2+\hasbound_\kvar} &= \frac 1 {\kvar} - \vertdiff_\kvar,
\end{align}
which yields
\begin{equation}
\label{eq: harmbound kvar eq}
\hasbound_\kvar = \frac {-\kvar + \sqrt{\kvar^2+1}} 2.
\end{equation}
By construction, the area of the shaded rectangle, which is $1/\kvar$, coincides with that of the trapezium shown with dashed line in the figure, whose upper side is the segment that joins the two intercept points of the curve. This segment is a chord of the curve, which is convex, and therefore the area under the curve between $x = \kvar-1$ and $x = \kvar$ is smaller than the shaded area:
\begin{equation}
\label{eq: int <}
\int_{\kvar-1}^\kvar \frac {\diff x}{x+1/2+\hasbound_\kvar} < \frac 1 \kvar.
\end{equation}
The terms $\hasbound_\nsuc, \ldots, \hasbound_{\nsam-1}$ as given by \eqref{eq: harmbound kvar eq} form a decreasing sequence. Thus defining $\hasbound = \hasbound_\nsuc$, as in \eqref{eq: hasbound eq}, ensures that $\hasbound > \hasbound_\kvar$ for any $\kvar > \nsuc$, and therefore, using \eqref{eq: int <},
\begin{equation}
\begin{split}
\int_{\nsuc-1}^{\nsam-1} \frac {\diff x}{x+1/2+\hasbound} &= \log \frac {\nsam-1/2+\hasbound}{\nsuc-1/2+\hasbound} \\
&\leq \sum_{\kvar=\nsuc}^{\nsam-1} \int_{\kvar-1}^{\kvar} \frac {\diff x}{x+1/2+\hasbound_\kvar}
< \sum_{\kvar=\nsuc}^{\nsam-1} \frac 1 \kvar = \harm_{\nsam-1} - \harm_{\nsuc-1},
\end{split}
\end{equation}
which proves the first inequality in \eqref{eq: lema diff harm: bound}.

Lastly, rewriting \eqref{eq: hasbound eq} as
\begin{equation}
\hasbound = \frac {-\nsuc + \nsuc \sqrt{1 + 1/\nsuc^2}} 2
\end{equation}
and using the inequality $1 < \sqrt{1+x} < 1+x/2$, $x \neq 0$ yields \eqref{eq: harmbound bound}.
\end{proof}

\begin{proof}[Proof of Theorem~\ref{theo: vahlogp}]
Consider $\nsuc \in \mathbb N$, $\nsuc \geq 2$ and $p \in (0,1)$, and let $\hasbound$ be defined as in Lemma \ref{lem: diff harm}. Noting that $\log((\nsam-1/2-\hasgen)/(\nsuc-1/2-\hasgen))$ is a continuous function of $\hasgen \in [0, \hasbound]$, for $\nsam > \nsuc$ it stems from the lemma that there is $\hasgen_\nsam \in (0, \hasbound)$ such that
\begin{equation}
\harm_{\nsam-1} - \harm_{\nsuc-1} = \log \frac {\nsam-1/2+\hasgen_\nsam} {\nsuc-1/2+\hasgen_\nsam},
\end{equation}
and for $\nsam = \nsuc$ this obviously holds as well. Therefore, $\Var[\vahlogp]$ can be written as
\begin{equation}
\label{eq: Var vahlogp 1}
\Var[\vahlogp] = \E[(-\harm_{\vansam-1}+\harm_{\nsuc-1}-\log p)^2] = \E\left[ \log^2 \frac {\nsuc-1/2+\hasgen_\vansam} {(\vansam-1/2+\hasgen_\vansam)p} \right],
\end{equation}
where $\hasgen_\vansam$ is a random variable that takes values in $(0, \hasbound)$. For convenience, let
\begin{equation}
\label{eq: hasbound prime}
\hasbound' = \frac 1 {4\nsuc}.
\end{equation}

Making use of the inequality
\begin{equation}
\label{eq: Topsoe}
\log ^2 x < \frac 1 x + x - 2,\ \ \ x > 0, 
\end{equation}
which results from \citet[equation~(14)]{Topsoe07}, and taking into account that $0 < \hasgen_\vansam < \hasbound < \hasbound'$, \eqref{eq: Var vahlogp 1} becomes
\begin{align}
\Var[\vahlogp] &< \E\left[ \frac {(\vansam-1/2+\hasgen_\vansam)p} {\nsuc-1/2+\hasgen_\vansam} \right] + \E\left[ \frac {\nsuc-1/2+\hasgen_\vansam} {(\vansam-1/2+\hasgen_\vansam)p} \right] - 2 \\
\label{eq: Var vahlogp 2}
&< \E\left[ \frac {(\vansam-1/2)p} {\nsuc-1/2} \right] + \E\left[ \frac {\nsuc-1/2+\hasbound'} {(\vansam-1/2+\hasbound')p} \right] - 2.
\end{align}
Computing the first summand in \eqref{eq: Var vahlogp 2} is straightforward, in view of \eqref{eq: neg bin: E vansam}. For the second, an upper bound is obtained next. It can be easily checked that, for $\nsam > 1$ and $|\hasbound'| < 1/2$,
\begin{equation}
\label{eq: inv minus 1/2 <}
\frac 1 {\nsam-1/2+\hasbound'} < \frac{1/2-\hasbound'}{\nsam-1} + \frac{1/2+\hasbound'}{\nsam} .
\end{equation}
Using \eqref{eq: neg bin: E inv minus 1} and the inequality \citep[section~II]{Mendo06}
\begin{equation}
\E \left[ \frac 1 {\vansam} \right] < \frac p {\nsuc-1+p},
\end{equation}
it follows from \eqref{eq: inv minus 1/2 <} that
\begin{equation}
\label{eq: Var vahlogp 2, second summand <}
\E\left[ \frac 1 {(\vansam-1/2+\hasbound')p} \right]
< \frac{1/2-\hasbound'}{\nsuc-1} + \frac{1/2+\hasbound'}{\nsuc-1+p} = \frac{\nsuc-1 + (1/2-\hasbound')p}{(\nsuc-1)(\nsuc-1+p)}.
\end{equation}
Thus, using \eqref{eq: neg bin: E vansam} and \eqref{eq: Var vahlogp 2, second summand <} into \eqref{eq: Var vahlogp 2},
\begin{equation}
\label{eq: Var vahlogp 3}
\begin{split}
\Var[\vahlogp] &< \frac {\nsuc-p/2} {\nsuc-1/2} + \frac{(\nsuc-1/2+\hasbound') (\nsuc-1 + (1/2-\hasbound')p)}{(\nsuc-1)(\nsuc-1+p)} - 2 \\
&= \frac {1-p} {2\nsuc-1} + \frac {(\nsuc-1)(1/2+\hasbound')(1-p) + (1/4-\hasbound'^2)p} {(\nsuc-1)(\nsuc-1+p)} \\
&= \left(\frac {1/2} {\nsuc-1/2} + \frac {1/2+\hasbound'} {\nsuc-1+p} \right)(1-p) + \frac {(1/4-\hasbound'^2)p} {(\nsuc-1)(\nsuc-1+p)}.
\end{split}
\end{equation}
Substituting \eqref{eq: hasbound prime},
\begin{equation}
\label{eq: Var vahlogp 4}
\begin{split}
\Var[\vahlogp]
&< \left(\frac {1/2} {\nsuc-1/2} + \frac {1/2+1/(4\nsuc)} {\nsuc-1+p} \right)(1-p) + \frac {(1/4-1/(16\nsuc^2))p} {(\nsuc-1)(\nsuc-1+p)} \\
&= \left( \frac 1 {\nsuc-1+p} + \frac {p/2-1/(8\nsuc)}{(\nsuc-1+p)(\nsuc-1/2)} \right)(1-p) + \frac {(1/4-1/(16\nsuc^2))p}
 {(\nsuc-1)(\nsuc-1+p)} \\
&< \frac 1 {\nsuc-1+p} \left( \left( 1 + \frac {p-1/(4\nsuc)} {2\nsuc-1} \right) (1-p) + \frac p {4(\nsuc-1)} \right),
\end{split}
\end{equation}
which gives \eqref{eq: E vahlogp bound}.

As a consequence of the above,
\begin{align}
\Var[\vahlogp]
&< \frac 1 {\nsuc-1} \left( \left( 1 + \frac {p} {2\nsuc-1} \right) (1-p) + \frac p {4(\nsuc-1)} \right) \\
\label{eq: Var vahlogp 5}
&< \frac 1 {\nsuc-1} \, \frac {2\nsuc^2(1-p)-\nsuc(3-9p/2+p^2)+1-9p/4+p^2} { (2\nsuc-1)(\nsuc-1) }.
\end{align}
The fraction with denominator $(2\nsuc-1)(\nsuc-1)$ in \eqref{eq: Var vahlogp 5} is easily seen to be less than $1$ for $\nsuc \geq 2$, $p \in (0,1)$, and thus \eqref{eq: E vahlogp bound guar} holds.
\end{proof}

Based on Theorem~\ref{theo: vahlogp}, the estimator of $\logodds = \log(p/(1-p))$ is defined in the following way (akin to the definition of $\vahodds$ in Section~\ref{part: odds}). Firstly, given $\nsuc \in \mathbb N$, $\nsuc \geq 2$, inverse binomial sampling with $\nsuc$ successes is applied to a sequence of \iid~Bernoulli random variables with parameter $p$. This requires a random number of samples $\vansam_1$, from which an estimate $\vahlogp_1$ of $\log p$ is obtained as in \eqref{eq: vahlogp}. Secondly, an analogous procedure, with the same $\nsuc$, is applied using new observations from the sequence of \iid~Bernoulli variables, except that the roles of successes and failures are exchanged. That is, samples are taken until $\nsuc$ \emph{failures} are observed. This consumes a random number of samples $\vansam_2$, and provides an estimate $\vahlogp_2$  of $\log (1-p)$. Then, the estimate $\vahlogodds$ of $\logodds$ is obtained as
\begin{equation}
\label{eq: vahlogodds}
\vahlogodds = \vahlogp_1 - \vahlogp_2 = -\harm_{\vansam_1-1} + \harm_{\vansam_2-1}.
\end{equation}
Let $\vansam = \vansam_1 + \vansam_2$ be the total number of samples used by the estimator $\vahlogodds$.

\begin{theorem}
\label{theo: vahlogodds}
For $\nsuc \in \mathbb N$, $\nsuc \geq 2$, $p \in (0,1)$, the estimator $\vahlogodds$ defined by \eqref{eq: vahlogodds} has the following properties:
\begin{align}
\label{eq: E vahlogodds unbiased}
\E[\vahlogodds] &= \log \frac p {1-p} \\
\label{eq: Var vahlogodds bound}
\Var[\vahlogodds] &< \frac { \nsuc^2-\nsuc/4-1/4 } { (\nsuc-1+p)(\nsuc-p)(\nsuc-1/2) } - \frac{p(1-p)}{(\nsuc-1/2)^2} \left( 1 - \frac{1}{2\nsuc-3} \right) \\
\label{eq: Var vahlogodds bound guar}
&<
\frac 1 {\nsuc-5/4} \left( 1 - \frac{7}{(4\nsuc-1)^2} \right) \\
\label{eq: vahlogodds sample size}
\E[\vansam] &= \frac{\nsuc}{p(1-p)}.
\end{align}
\end{theorem}

\begin{proof}
The equality \eqref{eq: E vahlogodds unbiased} follows from \eqref{eq: vahlogodds} and the fact that $\vahlogp_1$ and $\vahlogp_2$ are unbiased.

The bound \eqref{eq: E vahlogp bound} from Theorem~\ref{theo: vahlogp} can be rewritten as
\begin{equation}
\label{eq: E vahlogp bound 2}
\Var[\vahlogp] < \left(\frac 1 2 + \frac 1 {4\nsuc}\right) \frac{1-p}{\nsuc-1+p} + \frac {1-p} {2\nsuc-1} + \frac p {4(\nsuc-1+p)(\nsuc-1)}.
\end{equation}
Since $\vahlogp_1$ and $\vahlogp_2$ are statistically independent, $\Var[\vahlogodds] = \Var[\vahlogp_1] + \Var[\vahlogp_2]$. Applying \eqref{eq: E vahlogp bound 2} to bound $\Var[\vahlogp_1]$ and $\Var[\vahlogp_2]$, with $p$ replaced by $1-p$ for the latter, gives
\begin{equation}
\begin{split}
\Var[\vahlogodds] &< \left(\frac 1 2 + \frac 1 {4\nsuc}\right) \left(\frac{1-p}{\nsuc-1+p} + \frac{p}{\nsuc-p} \right) + \frac {1} {2\nsuc-1} \\
&\quad + \frac 1 {4(\nsuc-1)} \left( \frac p {\nsuc-1+p} + \frac {1-p} {\nsuc-p} \right) \\
&= \frac{\nsuc-2p(1-p)}{2(\nsuc-1+p)(\nsuc-p)} +\frac{\nsuc-2p(1-p)}{4\nsuc(\nsuc-1+p)(\nsuc-p)} + \frac {1} {2\nsuc-1} \\
&\quad + \frac {\nsuc-1+2p(1-p)}{4(\nsuc-1)(\nsuc-1+p)(\nsuc-p)}.
\end{split}
\end{equation}
Rearranging,
\begin{equation}
\begin{split}
\Var[\vahlogodds] &< \frac {1} {2\nsuc-1}  + \frac{\nsuc+1}{2(\nsuc-1+p)(\nsuc-p)} - \frac {p(1-p)} {(\nsuc-1+p)(\nsuc-p)} \left( 1 - \frac 1 {2\nsuc(\nsuc-1)} \right) \\
&= \frac{\nsuc^2-\nsuc/4-1/4}{(\nsuc-1+p)(\nsuc-p)(\nsuc-1/2)} \\
&\quad - \frac{p(1-p)} {(\nsuc-1+p)(\nsuc-p)} \left( 1 - \frac {1/2} {\nsuc-1/2} - \frac {1/2} {\nsuc(\nsuc-1)} \right),
\end{split}
\end{equation}
which combined with the inequalities $(\nsuc-1+p)(\nsuc-p) \leq (\nsuc-1/2)^2$ and
\begin{equation}
\frac 1 {\nsuc-1/2} + \frac 1 {\nsuc(\nsuc-1)} = \frac{\nsuc^2-1/2}{\nsuc(\nsuc-1/2)(\nsuc-1)} < \frac{\nsuc^2-1/4}{\nsuc(\nsuc-1/2)(\nsuc-1)} < \frac 1 {\nsuc-3/2}
\end{equation}
produces \eqref{eq: Var vahlogodds bound}.

Differentiating with respect to $p$, the right-hand side of \eqref{eq: Var vahlogodds bound} is seen to be decreasing for $p <1/2$ and increasing for $p>1/2$. Therefore its supremum occurs, by symmetry, when $p$ tends to either $0$ or $1$. This gives
\begin{equation}
\begin{split}
\Var[\vahlogodds] &< \frac{\nsuc^2-\nsuc/4-1/4}{\nsuc(\nsuc-1/2)(\nsuc-1)} = \frac{\nsuc^3-3\nsuc^2/2+\nsuc/16+5/16}{\nsuc(\nsuc-1/2)(\nsuc-1)(\nsuc-5/4)} \\
&= \frac 1 {\nsuc-5/4} \left( 1 - \frac{7\nsuc-5}{16\nsuc(\nsuc-1/2)(\nsuc-1)} \right) \\
&< \frac 1 {\nsuc-5/4} \left( 1 - \frac{7}{16\nsuc(\nsuc-1/2)} \right)
< \frac 1 {\nsuc-5/4} \left( 1 - \frac{7}{16(\nsuc-1/4)^2} \right) \\
&= \frac 1 {\nsuc-5/4} \left( 1 - \frac{7}{(4\nsuc-1)^2} \right),
\end{split}
\end{equation}
in accordance with \eqref{eq: Var vahlogodds bound guar}.

Lastly, $E[\vansam] = E[\vansam_1] + E[\vansam_2] = \nsuc/p + \nsuc/(1-p) = \nsuc/(p(1-p))$, which proves \eqref{eq: vahlogodds sample size}.
\end{proof}

By Theorem~\ref{theo: vahlogodds}, $\vahlogodds$ is an unbiased estimator of $\log(p/(1-p))$. Furthermore, choosing $\nsuc$ large enough guarantees, according to \eqref{eq: Var vahlogodds bound guar}, an error variance as low as desired irrespective of $p$; at the cost of an increase in average sample size proportional to $\nsuc$, as given by \eqref{eq: vahlogodds sample size}. Thus, for example, to ensure an RMSE smaller than $0.2$ it suffices to use $\nsuc = 27$ (which actually guarantees that the RMSE is less than $0.1971$). Reducing the target RMSE by a factor of $1/\sqrt{2}$, or the MSE by half, requires $r = 52$.

As in Section~\ref{part: odds}, the efficiency of $\vahlogodds$ can be used as a measure of quality of this estimator. Particularizing \eqref{eq: effic vahparamgen} for $\paramgen = \logodds$ and $\vahparamgen = \vahlogodds$, which gives $\diff \paramgen/\diff p = 1/(p(1-p))$, and making use of \eqref{eq: Var vahlogodds bound}--\eqref{eq: vahlogodds sample size},
\begin{align}
\label{eq: effic vahlogodds eq}
\effic(\vahlogodds) &= \frac{1}{p(1-p) \Var[\vahodds] \E[\vansam]} \\[1mm] 
\label{eq: effic vahlogodds bound 1}
& > \frac {1} { \displaystyle \frac { \nsuc^3-\nsuc^2/4-\nsuc/4 } { (\nsuc-1+p)(\nsuc-p)(\nsuc-1/2) } - \frac{\nsuc p(1-p)}{(\nsuc-1/2)^2} \left( 1 - \frac{1}{2\nsuc-3} \right) } \\[1mm] 
\label{eq: effic vahlogodds bound 2}
&> \frac{\nsuc - 5/4} {\nsuc} \,  \frac {(4\nsuc-1)^2}{(4\nsuc-1)^2-7} > \frac{\nsuc - 5/4} {\nsuc}.
\end{align}
This shows that, irrespective of $p$, the efficiency of $\vahlogodds$ is close to $1$ for moderate or large values of $\nsuc$, and tends to $1$ as $\nsuc$ grows.


The lower bound \eqref{eq: effic vahlogodds bound 1} is compared in Figure~\ref{fig: logodds_effic} with efficiency values obtained by Monte Carlo simulation, using the same procedure as in Section~\ref{part: odds}. As can be seen, the bound is less tight in this case, compared with the bound \eqref{eq: effic vahodds bound 1} for $\effic(\vahodds)$. This could be expected, because the expression of $\vahlogp$, on which $\vahlogodds$ is based, is more difficult to deal with than that of $\vahodds$ (in the proof of Theorem~\ref{theo: vahlogp}, $\Var[\vahlogp]$ is expressed as the expected value of the square of a certain logarithm, which is then bounded by a simpler function). Still, the uniform bound \eqref{eq: effic vahlogodds bound 2} is greater than the corresponding one for $\vahodds$, given by \eqref{eq: effic vahodds bound 2}; and, as was the case for that estimator, it becomes tighter as $\nsuc$ increases.

\begin{figure}%
\centering%
\includegraphics[width=.75\textwidth]{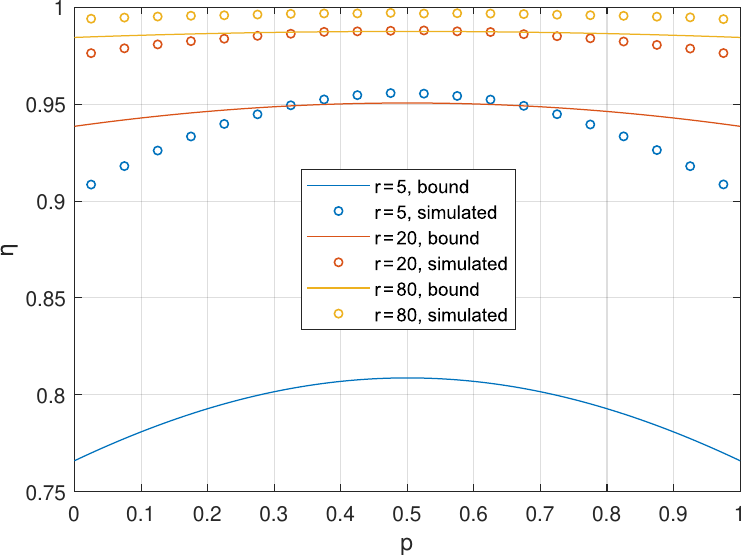}
\caption{Efficiency of log odds estimator $\vahlogodds$}%
\label{fig: logodds_effic}%
\end{figure}%

\section{Discussion and further work}
\label{part: disc}

Two sequential estimators have been proposed to estimate the odds $p/(1-p)$ and the log odds $\log(1/(1-p))$, respectively, from \iid~Bernoulli observations with parameter $p$. The estimators are unbiased, and guarantee that the relative variance (for odds) or the variance (for log odds) are below a desired value, irrespective of $p$. As far as the author is aware, no previously known estimators of the odds or log odds have this property. The guaranteed accuracy level is determined by the parameter $\nsuc$. Larger values of $\nsuc$ ensure better accuracy, at the cost of increasing the average sample size. Thus the parameter $\nsuc$ can be used to control the trade-off between these two magnitudes. The performance of the estimators is close to optimal in the sense of Wolfowitz's bound, with efficiency values greater than $(\nsuc-1)/(\nsuc+1)$ and $(\nsuc-5/4)/\nsuc$ for each estimator respectively.

The problem of estimating the odds or log odds is relevant in many areas of science and engineering. The proposed estimators can be used in any application, as long as the observations are \iid~Bernoulli variables.

Each of the two estimators is based on applying inverse binomial sampling separately with parameters $p$ and $1-p$, where the latter is obtained by swapping the role of successes and failures in the Bernoulli samples. This approach might appear wasteful, because information about both $p$ and $1-p$ could in principle be inferred using only one of the two sample sets. However, since the two probabilities are complementary, it is never the case that the two separate procedures require a large average number of samples at the same time; and indeed performance turns out to be close to optimal, as stated.

On the other hand, using a single set of samples to infer $p/(1-p)$ or its logarithm would pose some problems. As follows from \citet[theorem~4.1]{DeGroot59},
there is no unbiased estimator of these parameters in inverse binomial sampling. For $(1-p)/p$ an unbiased estimator does exist; but then it is not possible to guarantee a relative MSE, because as $p$ grows the sample size decreases, when it should increase so as to reduce the MSE. By similar arguments, an unbiased estimator of $\log((1-p)/p)$ exists but is not useful either. For fixed-size sampling there are no unbiased estimators of any of these quantities \citep[section~4]{DeGroot59}; and in any case, with a fixed sample size it is not possible to guarantee a desired relative MSE for odds, or a desired MSE for log odds. Moreover, regardless of the rule used for deciding the sample size, using a single sample set can make the error analysis more difficult. Namely, if the estimator is expressed as the product or sum of two estimators (as has been done in Sections~\ref{part: odds} and \ref{part: log odds}), those two estimators are not statistically independent, because they are based on the same observations.

The approach used in this paper could in principle be extended to estimating the probability ratio (or relative risk) $p_1/p_2$ or the odds ratio $p_1(1-p_2)/(p_2(1-p_1))$, as well as their logarithms, for two populations with parameters $p_1$ and $p_2$. For instance, to estimate $p_1/p_2$ with guaranteed relative MSE, two inverse binomial sampling procedures would be used for estimating $p_1$ and $1/p_2$ respectively; and error bounds analogous to those derived in the paper are easy to obtain. However, an important limitation of using this method with two populations is that it does not provide control on the proportion of samples taken from each population. Although it is possible to choose different pairs of $\nsuc_1$, $\nsuc_2$ that guarantee approximately the same estimation accuracy, the choice of $\nsuc_1$, $\nsuc_2$ does not determine the distribution of $\vansam_1/\vansam_2$, or the ratio $\E[\vansam_1]/\E[\vansam_2]$, as these depend also on $p_1$, $p_2$. Thus, each value of the unknown ratio $p_1/p_2$ would require a different choice of $\nsuc_1$, $\nsuc_2$ in order to attain a given $\E[\vansam_1]/\E[\vansam_2]$ and a desired accuracy.

Having control on $\vansam_1/\vansam_2$, or at least on $\E[\vansam_1]/\E[\vansam_2]$, when sampling from two populations is often important in applications. For example, if observations are taken in pairs, one sample from each population, then necessarily $\vansam_1 = \vansam_2$. Further work is needed to devise a method that can incorporate this type of restriction while guaranteeing a given accuracy. One possible approach is to use two-stage sampling:
a small sample set of each population is taken to produce an initial, rough estimate of $p_1/p_2$; using that information the  sizes of a second sample set of each population are obtained (for example, using two inverse binomial sampling procedures with $\nsuc_1$, $\nsuc_2$ computed from the estimated $p_1/p_2$); and from these new sample sets the final estimation of the parameter of interest is obtained.

Another generalization of this work would be to allow for the observations to be statistically dependent. In this more general setting, a natural assumption is that the sequence of observations is exchangeable. This means that its joint distribution of any order is invariant to permutations, i.e.~the statistical properties of the sequence are not affected by changes in the position of its elements. Then, the correlation between any pair of samples is the same, and is always non-negative \citep{Kingman78}.
In these conditions, the estimators proposed in this paper are not necessarily unbiased, and their variance is expected to be larger than in the \iid~case, because, for the same sample size, positively correlated observations are ``less informative'' than independent ones. Alternatively, dependence between samples can be modelled by assuming that the Bernoulli variables form a Markov chain, where the transition matrix is determined by the usual parameter $p$ and an additional parameter that controls the correlation between consecutive samples \citep{Lindqvist78}. This will, again, introduce bias and affect the variance of the estimators.


\end{document}